\documentclass[11pt]{amsart} 

\usepackage[utf8]{inputenc} 

\usepackage{geometry} 
\usepackage{graphicx} 
\usepackage{color}

\usepackage{booktabs} 
\usepackage{array} 
\usepackage{paralist} 
\usepackage{verbatim} 
\usepackage{subfig} 

\usepackage{amsfonts}
\usepackage{amsmath}
\usepackage{amsthm}
\usepackage{amssymb}
\usepackage{mathtools}
\mathtoolsset{centercolon,showonlyrefs,showmanualtags}

\newcommand{\thetak}{\theta^{\kappa}}

\newcommand{\mel}{\MoveEqLeft}
\usepackage{amsthm}

\newcommand{\D}{\mathcal{D}}

\newtheorem{example*}{Example\textsuperscript{*}}
\newtheorem{proposition*}{Proposition\textsuperscript{*}}

\newtheorem{corollary*}{Corollary\textsuperscript{*}}
\newtheorem{theorem}{Theorem}
\newtheorem{proposition}{Proposition}

\def\Limes#1#2 {\lim\limits_{#1\rightarrow #2}}

\def\eps{\epsilon}

\def\R{\mathbb{R}}

\def\T{\mathbb{T}}
\def\Z{\mathbb{Z}}
\DeclareMathOperator{\loc}{loc}

\def\N{\mathbb{N}}
\DeclareMathOperator{\spt}{spt}
\def\Xint#1{\mathchoice
{\XXint\displaystyle\textstyle{#1}}%
{\XXint\textstyle\scriptstyle{#1}}%
{\XXint\scriptstyle\scriptscriptstyle{#1}}%
{\XXint\scriptscriptstyle\scriptscriptstyle{#1}}%
\!\int}
\def\XXint#1#2#3{{\setbox0=\hbox{$#1{#2#3}{\int}$ }
\vcenter{\hbox{$#2#3$ }}\kern-.59\wd0}}
\def\avint{\Xint-}

\def\intT{\int_{\T^d}}
\def\intR{\int_{\R^d}}

\newtheorem{lemma}{Lemma}
\def\norm#1{\lVert #1 \rVert}
\def\scalar#1#2{\langle #1,#2 \rangle}

\renewcommand{\div}{\operatorname{div}}

\def\dd{\mathrm{d}}
\def\dx{\mathrm{d}x}
\def\dy{\mathrm{d}y}
\def\dt{\mathrm{d}t}
\def\ds{\mathrm{d}s}
\def\dr{\mathrm{d}r}

\newcommand{\grad}{\nabla}
\newcommand{\laplace}{\Delta}

\usepackage[square]{natbib}
\setcitestyle{numbers}

\usepackage{hyperref}

\begin{document}

\title[Propagation of regularity for transport equations]{Propagation of regularity for transport equations. A~Littlewood--Paley approach}
\thanks{This work is funded by the Deutsche Forschungsgemeinschaft (DFG, German Research Foundation) under Germany's Excellence Strategy EXC 2044 --390685587, Mathematics M\"unster: Dynamics--Geometry--Structure, and in particular under grant number 432402380.
\copyright 2022 by the authors.
}
\date{\today}

\author{David Meyer}
\address{DM: Institut f\"ur Analysis und Numerik, Westf\"alische Wilhelms-Universit\"at M\"unster, M\"unster, Germany}
\email{dmeyer2@uni-muenster.de}

\author{Christian Seis}
\address{CS: Institut f\"ur Analysis und Numerik, Westf\"alische Wilhelms-Universit\"at M\"unster, M\"unster, Germany}
\email{seis@wwu.de}

\begin{abstract}
It is known that   linear advection equations with Sobolev   velocity fields  have very poor regularity properties: Solutions  propagate only  derivatives of logarithmic order, which can be measured in terms of suitable Gagliardo seminorms. We propose a new approach to the study of regularity that is based on Littlewood--Paley theory, thus measuring regularity in terms of Besov norms. We recover the  results that are available in the literature and extend these optimally to the diffusive setting. As a consequence, we derive sharp bounds on  rates of convergence in the zero-diffusivity limit.
\end{abstract}
\maketitle

\section{Introduction}
Given a velocity field $u=u(t,x)\in\R^d$ on the $d$-dimensional flat torus $\T^d=[0,2\pi)^d$, we consider solutions $\theta=\theta(t,x)\in\R$ to the advection equation
\begin{equation}
\label{eq25} \partial_t \theta + u\cdot \grad \theta = 0
\end{equation}
on $\T^d$.
By considering this linear model it is supposed that the velocity field has no or negligible  feedback on the transported quantities, and $\theta$ is accordingly commonly referred to as a \emph{passive scalar} or simply \emph{tracer}. In spite of its mathematical simplicity, advection equations are of fundamental importance in a variety of models in physics. Motivated by applications in fluid dynamics, we suppose that $u$ is divergence-free, 
\begin{equation}\label{21}
\div u=0,
\end{equation}
which is relevant for incompressible fluids, and we assume that the velocity has low regularity properties, $u\in L^1((0,\infty);W^{1,p}(\T^d))$ for some $p\in(1,\infty)$. 
In this case, the advection equation \eqref{eq25}   can be rewritten in conservative form, which implies that the evolution preserves the mean of the  solutions, and this can be set to zero without restriction: If we denote by $\theta_0$ the initial datum of  $\theta$, it holds that
\begin{equation}
\label{22}
\intT \theta(t,x)\, \dx = \intT \theta_0(x)\, \dx =0,
\end{equation}
for any $t>0$. In fact, the theory of renormalized solutions by DiPerna and Lions \cite{DiPernaLions89} guarantees that the evolution \eqref{eq25} leaves any $L^q$-norm invariant, 
\begin{equation}\label{23}
\|\theta(t)\|_{L^q} = \|\theta_0\|_{L^q},
\end{equation}
for any $q\in [1,\infty]$. We shall assume that our solutions are bounded, $q=\infty$, in the sequel.

Our main concern in the present work is the maximal  regularity that is propagated by  solutions to \eqref{eq25}. This issue was addressed earlier: While in the Lipschitz case ($p=\infty$), it is easily seen that Sobolev regularity is propagated by solutions,
\[
\|\grad \theta(t)\|_{L^q} \le \exp\left(\int_0^t \|\grad u\|_{L^{\infty}}\, \dt\right) \|\grad\theta_0\|_{L^q},
\]
for any $q\in[1,\infty]$, there is a dramatic loss of regularity in the case of Sobolev vector fields ($p<\infty$). Indeed, it is possible for smooth initial data to lose all (even fractional) Sobolev regularity instantaneously at $t=0+$, as proved recently by Alberti, Crippa and Mazzucato \cite{AlbertiCrippaMazzucato19} (see also the construction by Crippa, Mazzucato, Iyer and Elgindi in \cite{crippa2021growth}). Earlier results in this direction were obtained by Jabin \cite{Jabin16}. Instead, during the evolution \eqref{eq25}, only a logarithm of a derivative can be preserved. On the level of the Lagrangian flow for the vector field $u$, such an observation is already implicitly contained in the work of Crippa and De Lellis \cite{CrippaDeLellis08}, who establish sharp regularity and stability estimates for the associated ODE. Their findings were recently translated to the PDE setting \eqref{eq25} by Bru\`e and Nguyen \cite{BrueNguyen21}, who measure the regularity propagated by solutions in terms of ($\R^d$-versions of) the Gagliardo seminorms
\[
\|\theta\|_{H^{\log, a}}: = \left(
\intT\intT \frac{|\theta(x) - \theta(y)|^2}{|x-y|^d} \log^{2a-1}\left(1+\frac1{|x-y|}\right)\, \dx\dy\right)^{1/2}.
\]
Similar seminorms were considered earlier by Ben Belgacem and Jabin \cite{BenBelgacemJabin13,BenBelgacemJabin19}, Bresch and Jabin \cite{BreschJabin18} and Leger \cite{Leger18}, in order to study the regularity properties of transport equations.  Function spaces of logarithmic smoothness are of independent interest in functional analysis and investigated, for instance, by Cobos, Dom\'inguez, Triebel and Tikhonov \cite{CobosDominguez15,Dominguez17,DominguezTikhonov18,cobos2016characterizations}.

The optimal results by Bru\`e and Nguyen \cite{BrueNguyen21} show that the $H^{\log, a}$ regularity is preserved under the linear evolution \eqref{eq25} if $
a = p/2$. 
In the present paper, we revisit this property by analyzing the propagation of regularity in terms of \emph{equivalent} Besov norms. More specifically, we consider the seminorms
\begin{equation}
\label{29}
\norm{\theta}_{B^{\log, a}}:=\left(\sum_{k=1}^\infty k^{2a} \norm{\theta*\phi_k}_{L^2}^2\right)^{1/2},
\end{equation}
where $\{\phi_k\}_{k\in\N}$ is a family of $L^1$-dilations which project onto dyadic Fourier blocks concentrated around  wave numbers of scale $ 2^k$, so that $\Delta_k =\Delta_k^{\phi} = \phi_k\ast$ is the standard Littlewood--Paley operator, and 
\[
\theta = \sum_{k=1}^{\infty} \theta\ast\phi_k,
\]
is the Littlewood--Paley decomposition of the (mean-free) function $\theta$. This norm is equivalent to the above Gagliardo seminorm. We will provide a detailed definition of the family  $\{\phi_k\}_{k\in\N}$ and discuss its properties in Section \ref{S2} below. The space of functions for which the above norm is finite will be denoted by $B^{\log,a}(\T^d)$. 

In our main result, we reproduce the bound in \cite{BrueNguyen21} by using Littlewood--Paley theory.

\begin{theorem}\label{T1}
Let $p\in(1,\infty)$ be given and let $u\in L^1((0,\infty);W^{1,p}(\T^d))$ be a divergence-free  vector field. Let $\theta\in L^{\infty}((0,\infty)\times \T^d)$ be a mean-free solution to the advection equation \eqref{eq25}. If $\theta_0\in B^{\log,a}(\T^d)$ for some
\[
\frac12\le  a < \frac{p}2,
\]
then $\theta\in L^{\infty}_{\loc}((0,\infty); B^{\log,a}(\T^d))$, and there exists a constant $C$, depending on $d$, $p$ and $a$, such that \begin{align}
\norm{\theta(t)}_{B^{\log, a}}\leq C\left(\int_0^t\norm{\nabla u}_{L^p}\dt\right)^{a}\norm{\theta_0}_{L^\infty}+C\norm{\theta_0}_{B^{\log, a}},
\end{align}
for any $t>0$. 
\end{theorem}

The reader who is familiar with DiPerna and Lions's well-posedness theory \cite{DiPernaLions89} for the advection equation \eqref{eq25} will notice that using Littlewood--Paley theory is quite natural in the study of regularity properties. Indeed, roughly speaking, the energy balance in \eqref{23}, from which uniqueness for the linear problem follows immediately, is derived by a mollification procedure: If $\{\phi^{\eps}\}_{\eps>0}$ is a family of standard mollifiers, the smooth function $\theta\ast \phi^{\eps}$ solves the equation
\[
\partial_t (\theta\ast\phi^{\eps}) + u\cdot \grad(\theta\ast\phi^{\eps}) = \div[ u, \phi^{\eps}*]\theta,
\]
where $[ u, \phi^{\eps}*]$ denotes the commutator of the operations ``multiply by $u$'' and ``convolute with $\phi^{\eps}$''. The key lemma in \cite{DiPernaLions89} provides the convergence to zero of the commutator in the limit $\eps\to0$. In view of the fact that a mollification in physical space $\theta\ast \phi^{\eps}$ corresponds to a cut-off in Fourier space $\theta\ast \phi_k$, our approach to regularity in Theorem~\ref{T1} is a \emph{sharp quantification of DiPerna and Lions's commutator lemma}. A crude estimate on the Littlewood--Paley commutator was used earlier by one of the authors to estimate the $L^1$-based energy spectrum in the diffusive setting \cite{Seis20}.

The approach to regularity for the transport equation that we propose in the present paper is   \emph{new} and at the same time  elementary. As we will see later, it also provides new tools for the derivation of   mixing bounds. Earlier bounds on regularity (and mixing) \cite{CrippaDeLellis08,BenBelgacemJabin13,Seis13a,BenBelgacemJabin19,BreschJabin18,BrueNguyen21} rely on estimates for the Hardy--Littlewood maximal function or sophisticated harmonic analysis tools (cf.~\cite{SeegerSmartStreet19}) in order to control difference quotients of the velocity field. Instead, our estimates   against the velocity gradient are based on   Littlewood--Paley theory. Actually, while earlier limitations to the case $p>1$ came from the fact that  strong type estimates on the maximal function operator cease to hold in $L^1$, our restriction to $p>1$ (and also to $a<p/2$) is related to  the failure of the Littlewood--Paley theorem (see \eqref{24}) in the endpoint cases $L^1$ and $L^{\infty}$.  Notice that by interpolation between the estimates for $a=1/2$ and $a=0$ given in \eqref{23}, our result can be extended to any $a\in[0,p/2)$ under suitable conditions on the initial datum.

An advantage of our method is that the above regularity theorem easily extends to the diffusive setting, in contrast to \cite{BrueNguyen21} as claimed in \cite{BrueNguyen21a}. Given a positive diffusivity constant $\kappa$, we denote by    $\thetak = \thetak(t,x)\in\R$  the solution to the advection-diffusion equation
\begin{equation}
\label{1}
\partial_t \thetak + u\cdot \grad \thetak = \kappa \laplace \thetak
\end{equation}
on $\T^d$, and we suppose that $\theta_0$ is   its mean-free initial datum, so that \eqref{22} carries over to $\thetak$. Thanks to diffusive dissipation, the energy identity in \eqref{23} becomes
\begin{equation}
\label{30}
\|\thetak(t)\|_{L^q} \le \|\theta_0\|_{L^q}.
\end{equation}
We will sometimes also make use of the balance law
\begin{equation}
\label{36}
\|\thetak(t)\|_{L^2}^2 + 2\kappa \int_0^t \|\grad\thetak\|_{L^2}^2\, \dt = \|\theta_0\|_{L^2}^2.
\end{equation}

The diffusive version of Theorem \ref{T1} is the following theorem.

\begin{theorem}\label{T2}
Let $p\in(1,\infty)$ be given and let $u\in L^1((0,\infty);W^{1,p}(\T^d))$ be a divergence-free vector field. Let $\theta\in L^{\infty}((0,\infty)\times \T^d)$ be a mean-free solution to the advection-diffusion equation \eqref{1}. If $\theta_0\in B^{\log,a}(\T^d)$ for some
\[
\frac12 \le a <\frac{p}2,
\]
then $\thetak\in L^{\infty}_{\loc}((0,\infty); B^{\log,a}(\T^d))$ with $\grad \thetak \in L^2_{\loc}((0,\infty);B^{\log, a}(\T^d))$, and there exists a constant $C$, depending on $d$, $p$ and $a$, such that \begin{align}
\norm{\thetak(t)}_{B^{\log, a}} +  \left(\kappa \int_0^t \|\grad \thetak\|_{B^{\log,a   }}^2\, \dt\right)^{1/2} \leq C\left(\int_0^t\norm{\nabla u}_{L^p}\dt\right)^{a}\norm{\theta_0}_{L^\infty}+ C\norm{\theta_0}_{B^{\log, a}},
\end{align}
for any $t>0$.  
\end{theorem}

A gain of regularity due to diffusion is not surprising, and the statement is certainly not optimal when asking for the maximal regularity that is propagated by \eqref{1}. The strength of Theorem \ref{T2} relies on the \emph{uniformity in $\kappa$} of the regularity estimate. Indeed, based on the new estimate, we are able to derive optimal bounds on  rates of convergence in the zero-diffusivity limit.

\begin{theorem}
\label{T4}
Under the hypotheses of Theorems \ref{T1} and \ref{T2}, there exists a constant $C$ dependent on $d$, $p$ and $a$  such that
\begin{equation}\label{26}
\begin{aligned}
\mel \|\theta(t) - \thetak(t)\|_{L^2} + \left(\kappa \int_0^t \|\grad\thetak\|_{L^2}^2\, \dt\right)^{1/2}\\
& \le \frac{C}{\log^a\left(2+\frac1{\kappa t}\right)}\left( \left(1+ \left(\int_0^t \|\grad u\|_{L^p}\, \dt\right)^p\right) \|\theta_0\|_{L^{\infty}} + \|\theta_0\|_{B^{\log, a}}\right),
 \end{aligned}
\end{equation}
for any $t,\kappa>0$.
\end{theorem}
Our result was conjectured by Bru\`e and Nguyen in \cite{BrueNguyen21a}, who, in the range $p>2$, on the one hand establish \eqref{26} for $a =\frac{p-2}2$   and on the other hand prove that \eqref{26} fails for $a>\frac{p(p-1)}{2(p-2)}$. A rate of $\log^{\frac{1}{2}}$ was established independently in \cite{bonicatto2021advection}.  Earlier disccusions on estimates of the vanishing rate of the dissipation can be found in \cite{drivas2019anomalous}. 
Regarding the term that multiplies the logarithm on the right-hand side, we notice that we could produce the same quantities that appear in the estimates of Theorems \ref{T1} and \ref{T2} if we restrict to sufficiently small $\kappa $.

Our bound only describes the short time behavior. Long time estimates as for instance formulated in Conjecture 1.7 of \cite{drivas2019anomalous} remain an open question.

We finally remark that there is an intimate relation between propagation of regularity and estimates on mixing rates in fluid flows.  Mixing in the purely advective context \eqref{eq25} refers to the flow-induced homogenization process of the advected quantity $\theta$ in the sense that 
\[
  \theta(t)\to 0\quad\mbox{weakly as }t\to \infty.
\]
Mixing is often measured in terms of negative Sobolev norms \cite{MathewMezicPetzold05,Thiffeault12} and it can be proved that for velocity fields satisfying $\|\grad u\|_{L^p} \le 1$ for some $p\in(1,\infty]$, mixing cannot proceed faster than exponentially in time \cite{CrippaDeLellis08,LunasinLinNovikovMazzucatoDoering12,Seis13a,IyerKiselevXu14,Leger18}. Exponential lower bounds on mixing can be obtained immediately from regularity estimates as in Theorem \ref{T1} and a duality principle \cite{Leger18}. Indeed, a short computation akin to the proof of Proposition 1 in \cite{Leger18} reveals, for example, that 
\[
\|\theta\|_{L^2} \le C \exp\left(\frac{\|\theta\|_{B^{\log,a}}}{\|\theta\|_{L^2}}\right)^{1/a} \|\theta\|_{\dot H^{-1}},
\]
for some $C>0$, and the bound in Theorem \ref{T1} and the conservation of the $L^2$-norm imply that
\[
\|\theta(t)\|_{\dot H^{-1}} \ge c e^{-\lambda t},
\]
for some $\lambda>0$ and $c>0$. This decay behavior is sharp \cite{YaoZlatos17,AlbertiCrippaMazzucato19a,BBPS22}. In Lemma \ref{L5} below we will establish a similar interpolation-type estimate that involves logarithmic optimal transportation distances to establish the strong convergence estimate in Theorem \ref{T4}.  We emphasize  that the derivation of the mixing bound presented in this  paper constitutes a first approach to fluid mixing that is based on Fourier analysis.

In the diffusive setting \eqref{1}, the filamentation generated by the advective flow creates sharp gradients of $\thetak$, which are quickly dissipated. As a result, the intensity is decreased much faster than in a purely diffusive context --- a phenomenon usually referred to as enhanced dissipation \cite{ConstantinKiselevRyzhikZlatos08,CotiZelatiDelgandioElgindi20}. If 
\[
\|\thetak(t)\|_{L^2} \le e^{-Dt} \|\theta_0\|_{L^2}
\]
for some $D>0$ and any $t\ge 0$, the energy balance implies that
\[
\|\theta_0\|_{L^2}^2 \le 4\kappa \int_0^t \|\grad\thetak\|_{L^2}^2\, \dt ,
\]
for any $t\ge (\log2)/2D$, and the estimate on the dissipation rate in \eqref{26} enforces that $D\le  \log^{-1}1/\kappa$ for some $C>0$, which is optimal \cite{BBPS21}. This conditional  bound on enhanced dissipation was established in \cite{Seis20a} with the help of stability estimates \cite{Seis17,Seis18}. Suboptimal bounds based on  regularity are obtained in \cite{BrueNguyen21a}.

\medskip

\emph{Organization of the paper.} In Section \ref{S2} we give a precise definition of the Besov space of logarithmic smoothness considered in the present paper and fix the underlying Littlewood--Paley decomposition. We furthermore discuss elementary properties of these and provide first auxiliary functional analytic  results. In Section \ref{S3} we prove the main regularity results Theorem \ref{T1} and Theorem \ref{T2}. The final Section \ref{S4} is devoted to the zero-diffusivity estimate in Theorem \ref{T4}.

\medskip

\emph{Notation.} We will often write $f\lesssim g$ for two reals $f$ and $g$ if there exists a constant $C$ depending possibly on $d$, $p$, $a$ and other integrability exponents such that $f\le Cg$.

\section{Littlewood--Paley theory and logarithmic Besov space}\label{S2}

We will define Littlewood--Paley decompositions of periodic functions with the help of standard $L^1$-dilations on the full space. This makes it necessary, at least at this introductory level, to distinguish between Fourier transforms on $\T^d$ and on $\R^d$. We recall the respective definitions to fix conventions.
The Fourier transform of a function $\theta$ on $\T^d$  and its Fourier representation are defined and given by
\[
\widehat \theta(\eta) = \avint_{\T^d} e^{-i\eta\cdot x} \theta(x)\, \dx,\quad \theta(x) = \sum_{\eta\in \Z^d} e^{i\eta\cdot x}\widehat \theta(\eta),
\]
for any wave number $\eta\in \Z^d$ and any $x\in \T^d$. The Fourier transform of a  function $\phi$ on $\R^d$ and its Fourier representation are defined and given by
\[
\widehat \phi(\xi) = \frac1{(2\pi)^{\frac{d}2}} \int_{\R^d} e^{-i\xi\cdot x}\phi(x)\, \dx,\quad \phi(x) =  \frac1{(2\pi)^{\frac{d}2}}\int_{\R^d} e^{i\xi\cdot x}\widehat \phi(\xi)\, \mathrm{d}\xi,
\]
for any frequency $\xi\in \R^d$
and any $x\in \R^d$.

Convolutions will be always understood in $\R^d$, which makes it occasionally necessary to extend functions periodically from $\T^d$ to $\R^d$. The convolution of $\theta$ and $\phi$ as above is thus the function
\[
(\phi\ast\theta)(x) = \int_{\R^d} \phi(x-y)\theta(y)\, \mathrm{d}y = \intT\Big(\sum_{\eta\in\Z^d} \phi(x-z-2\pi\eta)\Big)\theta(z)\, \mathrm{d}z,
\]
which is periodic in $x$, and which we interpret thus as a function on $\T^d$. Its Fourier transform obeys then the  product rule,
\[
\widehat{(\phi\ast \theta)}(\eta) = (2\pi)^{d/2}\widehat \phi(\eta)\widehat \theta(\eta),
\]
for any $\eta\in \Z^d$.

We will now describe a standard way to construct Littlewood--Paley decompositions.  We consider a rotationally symmetric Schwartz function $\phi$ whose Fourier transform $\widehat \phi$ is supported in $B_1(0)$ and satisfies  $\widehat{\phi}(\xi)=1$ for $\xi\in \overline{B_{\frac{1}{2}}(0)}$. This function serves as a generator of a dyadic partition of unity in frequency space in the following way: For any $k\in\Z$, we define $\phi_k=2^{kd}\phi(2^k\cdot)-2^{(k-1)d}\phi(2^{k-1}\cdot)$, which has the property that its Fourier transform is supported on a dyadic annulus, 
\begin{equation}
\label{13}
\spt \widehat \phi_k \subset B_{2^{k}}(0)\setminus \overline{B_{2^{k-2}}(0)},
\end{equation}
and thus, each $\phi_k$ has overlapping frequency support only with its direct neighbors, so that
\begin{equation}
\label{27}
\phi_k = \phi_k\ast \left(\phi_{k-1}+\phi_k+\phi_{k+1}\right).
\end{equation}
Moreover,  the  family $\{\phi_k\}_{k\in \Z }$ is constructed in such a way that it partitions the frequency space
\[
 \sum_{k\in \Z} \widehat \phi_k =1\quad \mbox{on } \R^d\setminus \{0\}.
\]
In particular, if $\theta$ is a mean-zero \eqref{22} function on $\T^d$, it holds that $\widehat{\theta}(0)=0$, and thus 
\[
\theta = \sum_{k\in\N} \theta\ast \phi_k.
\]
This is the \emph{Littlewood--Paley decomposition} of the mean-zero function $\theta$ on $\T^d$. 

An important consequence of this decomposition and of the almost orthogonality property \eqref{27}  of the frequency projections is the Littlewood--Paley theorem,
\begin{equation}
\label{24}
\| \Big(\sum_k (\theta*\phi_k)^2\Big)^{1/2}\|_{L^q} \sim  \|\theta\|_{L^q} ,
\end{equation}
cf.~Theorem 6.1.2 in \cite{Grafakos}. It holds true for  integrability exponent $q\in(1,\infty)$. We will make use only of the upper bound and the $L^q$-control of the Littlewood--Paley square function, which holds true for more general $L^1$-dilations $\tilde \phi_k = 2^{kd}\tilde \phi_0(2^k\cdot)$ if $\tilde \phi_0$ is a mean-zero Schwartz function, 
\begin{equation}
\label{28}
\| \Big(\sum_k (  \theta*\tilde\phi_k)^2\Big)^{1/2}\|_{L^q} \lesssim  \|\theta\|_{L^q} ,
\end{equation}
cf.~Theorem 6.1.2 in \cite{Grafakos}.

 For any $k\in\N$, the part of a function $f$ whose frequencies are concentrated in the $k$-th dyadic annulus, will be denoted by $f_k$, more precisely, $f_k=f*\phi_k$.

We furthermore consider the low pass filters  $\psi_k:=\phi+\sum_{j=1}^k\phi_j=2^{kd}\phi(2^k\cdot)$ and  define with them the high frequency parts $f_k^\geq=f-f*\psi_{k-1} = \sum_{j=k}^{\infty} f_j$, where the second identity holds true provided that $f$ has zero mean, so that $f\ast\phi=0$. Similarly, we define $f_k^{\leq } = f\ast \psi_k$, so that $f  = f_k^\leq + f_{k+1}^\geq$. 

In the following we only consider mean-free $\theta$.

With these preparations at hand, the logarithmic Besov norm in \eqref{29},
\[
\norm{\theta}_{B^{\log, a}}:=\left(\sum_{k=1}^\infty k^{2a} \norm{\theta_k}_{L^2}^2\right)^{1/2},
\]
is well-defined for any $a\in \R$. Notice that since the frequency support of $\phi_k$  concentrates around the scale $2^k$, the prefactor scales logarithmically in the frequency, so that
\[
\sum_{k=1}^\infty k^{2a} \norm{\theta_k}_{L^2}^2 \sim \sum_{\eta\in \Z^d} \log^{2a}(|\eta|+1) |\widehat\theta(\eta)|^2,
\] 
which defines a Sobolev norm for a derivative of logarithmic order.

For accuracy we should note here that the associated function space is in the literature denoted by $B_{2,2}^{0,a}(\T^d)$, cf.~\cite{CobosDominguez15,Dominguez17,DominguezTikhonov18}. We also remark that this agrees with the space $H^{\log,2a}$, used by Bru\`{e} and Nguyen \cite{BrueNguyen21a}, as shown in Theorem 1.4 of \cite{brue2019sobolev}.

For our analysis, instead of working with the $B^{\log,a}$ norm, it is more convenient to consider the equivalent norm
\begin{align}\label{32}
\norm{\theta}_{\bold{B}^{\log a}}:=\left(\sum_{j=1}^\infty j^{2a-1} \|\sum_{k\geq j}^\infty \theta_k\|_{L^2}^2\right)^{1/2}=\left(\sum_{j=1}^\infty j^{2a-1} \|\theta_j^\geq\|_{L^2}^2\right)^{1/2}.
\end{align}
This kind of norm was introduced in \cite{cobos2016characterizations}.

The equivalence is established in the following lemma:

\begin{lemma}\label{lemma 1}
For any $a>0$, it holds that
\[
 \bold{B}^{\log ,a}(\T^d) =  B^{\log, a}(\T^d),
 \]
and the respective norms are equivalent.
 \end{lemma}
 
This type of  result is well-known (see e.g.\cite{CobosDominguez15}[Thm.\ 3.3]) in the theory of Besov spaces. As the above characterization is presumably less popular in the applied maths community, we provide an  elementary proof  for the convenience of the reader.

\begin{proof}
 The equivalence of the norms is a consequence of the fact that
 \begin{equation}
 \label{12}
 \|\sum_{k\ge j} \theta_k\|_{L^2}^2 \lesssim \sum_{k\ge j-1} \|\theta_k\|_{L^2}^2 \lesssim \| \sum_{k \ge j-2} \theta_{k}\|_{L^2}^2.
 \end{equation}
Indeed, considering, for instance, the first of the estimates, multiplying by $j^{2a-1}$ and summing over $j$ gives
\begin{align*}
\sum_{j\ge 1} j^{2a-1} \|\sum_{k\ge j} \theta_k\|_{L^2}^2  &\lesssim \sum_{j\ge 1} j^{2a-1} \sum_{k\ge j-1} \|\theta_k\|_{L^2}^2\\
& = \sum_{k\ge 1} \left(\sum_{j=1}^{k+1} j^{2a-1}\right)\|\theta_k\|_{L^2}^2\\
&\lesssim \sum_{k\ge 1} k^{2a} \|\theta_k\|_{L^2}^2,
\end{align*}
where we have used the fact that $\theta\ast\phi_0=0$ because $\theta$ was assumed to have zero mean. The converse estimate is established very similarly.

The verification of  \eqref{12} relies on the almost-orthogonality property of the Littlewood--Paley decomposition. Indeed, by the virtue of \eqref{27}, the phase blocks $\theta_k$ and $\theta_{\ell}$ are orthogonal unless $\ell\in\{k-1,k,k+1\}$. 
Therefore, we find on  the one hand by expanding the sums that
\begin{align*}
\|\sum_{k\ge j} \theta_k\|_{L^2}^2   = \sum_{k\ge j} \sum_{\ell = k-1,k,k+1} \intT \theta_k\theta_{\ell}\, \dx\lesssim \sum_{k\ge j-1}\| \theta_k\|_{L^2}^2 .
\end{align*} 
On the other hand, using \eqref{13} again, we may write 
\[
\|\theta_k\|_{L^2}^2  =  \| \sum_{\ell \ge j-1} \theta_k * \phi_{\ell}\|_{L^2}^2,
\]
and summation over $k$ gives
\[
\sum_{k\ge j} \|\theta_k\|_{L^2}^2  =  \sum_{k\ge j}\| \sum_{\ell \ge j-1} \theta_k * \phi_{\ell}\|_{L^2}^2 \le \sum_{k\ge 1}\| \Big(\sum_{\ell \ge j-1} \theta_{\ell}\Big)* \phi_k\|_{L^2}^2 \lesssim \| \sum_{\ell \ge j-1} \theta_{\ell}\|_{L^2}^2,
\]
where we have used the Littlewood--Paley characterization \eqref{24} of the $L^2$-norm. 
\end{proof} 

We will also require some interpolation identities between our norm and other quantities measuring logarithmic smoothness.

\begin{lemma} \label{L1}
Let $b, a\geq 0$ and $r\ge 2$ be given  and $\theta\in L^\infty(\T^d)\cap B^{\log ,a} $. Let $\eta_k=2^{kd}\eta(2^k\cdot)$ be a family of mean-free Schwartz functions for which $\widehat{\eta_1}$ is compactly supported.
\begin{itemize} 
\item[a)] If   $2a = br$, then  it holds \begin{equation}
\label{7}
\norm{\sup_{k\geq 0} k^{b}|\theta_k^\geq|}_{L^{r}}\lesssim\norm{\theta}_{L^\infty}^{1-\frac{b}{a}} \norm{\theta}_{B^{\log, a}}^{\frac{b}{a}}.
\end{equation}

\item[b)] If   $2a = br$, then  it holds \begin{equation}
\label{7a}
\norm{\sup_{k\geq 0} k^{b}|\theta_k|}_{L^{r}}\lesssim\norm{\theta}_{L^\infty}^{1-\frac{b}{a}} \norm{\theta}_{B^{\log, a}}^{\frac{b}{a}}.
\end{equation}

\item[c)] If $b<\frac{2a}r$ then
\begin{align}\label{15}
\norm{\left(\sum_{k\geq 1} k^{2b}|\theta*\eta_k|^2\right)^\frac{1}{2}}_{L^{r}}\lesssim\norm{\theta}_{L^{\infty}}^{1-\frac{b}{a}}\norm{\theta}_{B^{\log ,a}}^{\frac{b}{a}}.
\end{align}
\end{itemize}
\end{lemma}

We remark that establishing part c) for $b=\frac{2a}{r}$ would allow for $a=\frac{p}{2}$ in Theorem \ref{T1} and \ref{T2} (and subsequently also in Theorem \ref{T4}).

\begin{proof}
a) \& b) The arguments for \eqref{7} and \eqref{7a} are identical. We focus on the first one. By using that $\sum_{m=1}^km^{br-1}\simeq k^{br}$ and pulling the supremum inside the sum we obtain that
\begin{align}
&\norm{\sup_{k\geq 1} k^{b}|\theta_k^\geq|}_{L^{r}}^r\lesssim \norm{\sup_{k\geq 1}  \sum_{m=1}^k m^{br-1}|\theta_k^\geq|^{r} }_{L^{1}}\leq \norm{ \sum_{m\geq 1} m^{br-1}\sup_{k\geq m}|\theta_k^\geq|^{r}}_{L^{1}} .
\end{align}
Pulling the norm inside via the triangle inequality, we see that this can be further  estimated by 
\begin{align}
\norm{\sup_{k\geq 1} k^{b}|\theta_k^\geq|}_{L^{r}}^r\lesssim
\sum_{m\geq 1} m^{br-1}\norm{\sup_{k\geq m} |\theta_k^\geq|^{r}}_{L^{1}}=\sum_{m\geq 1} m^{br-1}\norm{\sup_{k\geq m} |\theta_k^\geq|}_{L^{r}}^r.
\end{align}
To estimate the suprema we note that, on the one hand, for $k\geq m$ we have that $\theta_k^\geq=(\theta_{m-2}^\geq)_k^\geq$ (with $\theta_0^\geq=\theta_{-1}^\geq=\theta$). On the other hand, we also have
\begin{align}\norm{\sup_{k\geq 1} |f_k^{\geq}|}_{L^r}=\norm{\sup_{k\geq 1} |f-f_{k-1}^\leq|}_{L^r}\leq \norm{f}_{L^r}+\norm{\sup_{k\geq 0} |f_k^\leq|}_{L^r}\lesssim \norm{f}_{L^r}\end{align} for all $f\in L^r$, where in the last step we have used the classical fact that the maximal function $\sup_{k\geq 1} f*\psi_{k-1}$ is bounded on $L^r$, see, e.g., Theorem 2.1.4 in  \cite{Grafakos_mod}. Applying the latter estimate to $\theta_{m-2}^\geq$ yields that
\[
\norm{\sup_{k\geq m} |\theta_k^\geq|}_{L^{r}}\lesssim \norm{\theta_{m-2}^\geq }_{L^r},
\]
and thus, inserting this estimate in the term above gives the upper bound
\begin{align} 
\norm{\sup_{k\geq 1} k^{b}|\theta_k^\geq|}_{L^{r}}^r\lesssim
\sum_{m\geq 1} m^{br-1}\norm{\theta_{m-2}^\geq}_{L^{r}}^{r}.
\end{align}
It remains to interpolate the $L^r$-norm   between $L^\infty$ and $L^2$, and we obtain the bound
\begin{align}
\norm{\sup_{k\geq 1} k^{b}|\theta_k^\geq|}_{L^{r}}^r  \lesssim  \sum_{m\geq 1} m^{br-1}\norm{ \theta_{m-2}^\geq}_{L^{2}}^{2}\norm{\theta_{m-2}^\geq}_{L^\infty}^{r-2}  \lesssim\norm{\theta}_{L^\infty}^{r-2} \sum_{m\geq 1} m^{br-1}\norm{ \theta_{m}^\geq}_{L^{2}}^{2},
\end{align}
which proves \eqref{7}, as desired.

\medskip

c) Using Hölder's inequality first on the sum and then on the norms we obtain \begin{equation}\label{eq 9}
\begin{aligned}
\norm{\Big(\sum_{k\geq 1} k^{2b}|\theta*\eta_k|^2\Big)^{\frac{1}{2}}}_{L^{r}}&\leq \norm{\Big(\sum_{k\geq 1} k^{2a}|\theta*\eta_k|^2\Big)^{\frac{b}{2a}}\Big(\sum_{k\geq 1} |\theta*\eta_k|^2\Big)^{\frac{a-b}{2a}}}_{L^{r}}\\
& \leq \norm{\Big(\sum_{k\geq 1} k^{2a}|\theta*\eta_k|^2\Big)^{\frac{1}{2}}}_{L^2}^{\frac{b}{a}}\norm{\Big(\sum_{k\geq 1} |\theta*\eta_k|^2\Big)^{\frac{1}{2}}}_{L^{s}}^{\frac{a-b}{a}},
\end{aligned}
\end{equation}
where we have set $s = {2(a-b)r}/{(2a-br)}$ for abbreviation. We notice that all exponents (including $s$) are finite provided that $b<\frac{2a}r$, which also guarantees that $a>b$ since $r\ge 2$ by assumption. 
Therefore, we may apply the Littlewood--Paley theorem to the norm on the right and we obtain with the help of Jensen's inequality that 
\begin{align}\label{eq 10}\norm{\Big(\sum_{k\geq 1} |\theta*\eta_k|^2\Big)^{\frac{1}{2} }}_{L^{s}}\lesssim\norm{\theta}_{L^{s} } \lesssim \norm{\theta}_{L^\infty}.
\end{align}

It remains to estimate the first term on the right-hand side of \eqref{eq 9}. Splitting into phase blocks and using almost orthogonality and Plancherel's theorem twice, we may estimate  \begin{align}
\norm{\Big(\sum_{k\geq 1} k^{2a}|\theta*\eta_k|^2\Big)^{\frac{1}{2}}}_{L^2}^2&\le \sum_{j\ge 1} \norm{\Big(\sum_{k\geq 1} k^{2a}|\theta_j *\eta_k|^2\Big)^{\frac{1}{2}}}_{L^2}^2\\
&\lesssim\sum_{j\ge 1} \sum_{k\geq 1} \sum_{\xi\in \Z^d}  k^{2a}|\widehat{\theta_j}(\xi)|^2|\widehat{\eta_k}(\xi)|^2\\
&\le  \sum_{j\geq 1}\norm{\theta_j}_{L^2}^2\sum_{k\geq 1} k^{2a}\sup_{\xi\in B_{2^j}(0)\backslash B_{2^{j-2}}(0)}|\widehat{\eta_k}(\xi)|^2.
\end{align}
If $R$ is such that the support of $\widehat{\eta_1}$ is contained in $B_R(0)$, then the supremum in the sum is only nonzero when $2^kR\geq 2^{j-2}$. In that case we may use that $\eta_k$ has zero mean, $\widehat{\eta_k}(0)=0$, and estimate the supremum with 
\begin{align}
|\widehat{\eta_k}(\xi)| \le 2^j |\grad \widehat{\eta_k}(\xi)| = 2^{j-k} |\grad \widehat{\eta_1}(2^{-k}\xi)| \le 2^{j-k}\|\grad\widehat{\eta_1}\|_{L^{\infty}} \lesssim 2^{j-k},
\end{align}
for any $\xi \in B_{2^j}(0)\setminus B_{2^{j-2}}(0)$.  Hence the inner sum can be estimated by
 \begin{align}
\sum_{k\ge \lfloor j-2- \ln_2 R\rfloor}k^{2a}2^{j-k}&\le \sum_{k\ge \lfloor j-2- \ln_2 R\rfloor}|k-j|^{2a}2^{j-k}  + j^{2a} \sum_{k\ge \lfloor j-2- \ln_2 R\rfloor} 2^{j-k} \\
&\lesssim 1+ j^{2a}\lesssim j^{2a}.
\end{align}
This yields 
\begin{align}
\label{eq 8}\norm{\Big(\sum_{k\geq 1} k^{2a}|\theta*\eta_k|^2\Big)^{\frac{1}{2}}}_{L^2}\lesssim \norm{\theta}_{B^{\log ,a }}.
\end{align}

Inserting \eqref{eq 10} and \eqref{eq 8} into \eqref{eq 9} yields the statement.\end{proof}

\section{Regularity estimates. Proofs of Theorems \ref{T1} and \ref{T2}}\label{S3}

In this section, we turn to the proofs of Theorems \ref{T1} and \ref{T2} simultaneously. The only difference is the occurrence of the diffusion term in \eqref{1}, which is, in fact, easily controlled in Besov norms. We thus study \eqref{1} rather than \eqref{eq25} and allow for $\kappa=0$ when we consider the latter. For notational simplicity, we will write $\theta = \thetak$, also for positive diffusivities. 

\begin{proof}[Proof of Theorems \ref{T1} \& \ref{T2}]
We notice that the evolution of the high frequency contributions $\theta_k^\geq$ can be written as a transport equation with  a  forcing term,
\[
\partial_t \theta_k^\geq +u\cdot \grad \theta_k^\geq  
 =\kappa \laplace \theta_k^{\ge} + [ \psi_{k-1}\ast,u\cdot ]\grad\theta,
\]
where the second term on the  right-hand side is the commutator of the operations ``multiply by $u$'' and ``convolute with $\psi_{k-1}$'', that we have already seen in the introduction. Thanks to the imposed regularity of  $u$ and the mollifying effect of the convolution, the high frequency part $\theta_k^\geq$ can be interpreted as a distributional solution, which is renormalized in the sense of DiPerna and Lions \cite{DiPernaLions89}. We may thus compute the rate of change of the variance in a straightforward way,
\begin{equation}
\label{2}
\frac12\frac{\text{d}}{\text{d}t}\intT (\theta_k^\ge)^2\dx + \kappa \intT |\grad\theta_k^{\ge}|^2\, \dx=\intT \theta_k^\ge [ \psi_{k-1}\ast,u\cdot ]\grad\theta\, \dx.
\end{equation}
By an integration by parts, we rewrite the commutator as
\[
[\psi_{k-1}\ast,u\cdot]\grad\theta(x) = \intR \theta(x-y) (u(x)-u(x-y))\cdot\grad\psi_{k-1}(y)\, \dy,
\]
and thus, \eqref{2} becomes
\begin{align*}
\mel
\frac12\frac{\text{d}}{\text{d}t}\intT (\theta_k^\geq)^2\dx  + \kappa \intT |\grad\theta_k^{\ge}|^2\, \dx\\
&= \intT\intR \theta_k^\ge(x)\theta(x-y) (u(x)-u(x-y))\cdot\grad\psi_{k-1}(y)\, \dy\dx.
\end{align*}
We now expand $\theta$ and $u$ into the phase blocks and use Parseval's identity to notice that
\[
\intT \theta_k^\geq u_m\cdot \grad \psi_{k-1} \ast \theta_{\ell}\, \dx = \sum_{\eta,\eta'\in\Z^d} {\widehat{\theta_k^\geq}(-\eta)} \widehat{u_m} (\eta-\eta') \cdot \widehat{\grad \psi_{k-1}} (\eta')\widehat{\theta_{\ell}}(\eta').
\]
Using the definition of the phase blocks, we observe that the summands are nonzero only if $ \ell\le k$ and vanish if $2^{k-3} \ge 2^m+2^{\ell}$ for some $\ell\le k-1$.  Similarly, we find that
\[
\intT \theta_k^\geq\div\left(\psi_{k-1}\ast(u_m\theta_{\ell})\right)\, \dx = \sum_{\eta,\eta'\in\Z^d} {\widehat{\theta_k^\geq}(-\eta)} \widehat{\grad\psi_{k-1}}(\eta) \cdot \widehat{u_m}(\eta')\widehat{\theta_{\ell}}(\eta-\eta')
\]
is nonzero only if $2^{\ell-2}\leq 2^{k-1}+2^m$. We thus conclude that both contributions are $0$ if $|\ell|> k+4$ and $|m|< k+3$ and hence
\begin{align}
\mel\frac{1}{2}\frac{\text{d}}{\text{d}t}\norm{\theta}_{\bold{B}^{\log, a}}^2 + \kappa \|\grad \theta\|_{\bold{B}^{\log, a}}^2\\
&=\frac12\frac{\text{d}}{\text{d}t}\sum_{k\geq 1}k^{2a-1}\intT (\theta_k^\ge)^2\dx  + \kappa \sum_{k\ge 1} k^{2a-1} \intT |\grad\theta_k^\ge|^2\dx\\
&=\sum_{k\geq 1}k^{\alpha} \intT\intR \theta_k^\geq(x)\theta(x-y) u_{k+3}^\geq(x)\cdot\grad\psi_{k-1}(y)\, \dy\dx\\
&\qquad-\sum_{k\geq 1}k^{\alpha}\intT\intR \theta_k^\geq(x)\theta(x-y)u_{k+3}^{\geq }(x-y)\cdot\grad\psi_{k-1}(y)\, \dy\dx\\
&\qquad+\sum_{k\geq 1}k^{\alpha} \intT\intR \theta_k^\geq(x)\theta_{k+4}^{\leq}(x-y) \left(u_{k+2}^\leq(x)- u_{k+2}^{\leq }(x-y)\right)\cdot\grad\psi_{k-1}(y)\, \dy\dx\\
&=:\,\mathrm{I} - \mathrm{II} + \mathrm{III},
\end{align}
where we have set $\alpha  =2a-1$.
We are going to show that  $\mathrm{I}$, $\mathrm{II}$  and $\mathrm{III}$ can all be estimated by
\begin{align}
\label{eq 4}
|\mathrm{I}|+ |\mathrm{II}| + |\mathrm{III}| \lesssim \norm{\nabla u}_{L^p}\norm{\theta}_{L^\infty}^{\frac{1}{a}}\norm{\theta}_{\bold{B}^{\log ,a}}^{\frac{2a-1}{a}},
\end{align}
as long as $a<\frac{p}{2}$, so that the above identity implies  the differential inequality
\begin{equation}
\label{31}
a \frac{\text{d}}{\text{d}t} \norm{\theta}_{\bold{B}^{\log, a}}^{\frac1{a}} + \kappa \| \theta\|_{\bold{B}^{\log,a}}^{\frac1a-2}\|\grad \theta\|_{\bold{B}^{\log,a}}^2 \lesssim \|\grad u\|_{L^p} \|\theta\|_{L^{\infty}}^{\frac1{a}}.
\end{equation}
An integration in time yields the desired estimate in Theorem \ref{T1} via Lemma \ref{lemma 1} 
 and the fact that the amplitude remains bounded, cf.~\eqref{23} with $q=\infty$. To obtain the additional control on the dissipation in the statement of Theorem \ref{T2}, we make use of the bound of Theorem \ref{T1} to estimate the prefactor in the dissipation term in \eqref{31}.

We now turn to the proof of \eqref{eq 4}.
To fix some notation, we write  $\alpha=\alpha_1+\alpha_2$ for some $\alpha_1,\alpha_2>0$ and we take $q,s$ such that $\frac{1}{p}+\frac{1}{q}+\frac{1}{s}=1$.

\medskip

\emph{Estimate of $\mathrm{I}$.} 
The argument in this case relies on the fact that the high frequencies of $u$ become small due to the control on $\nabla u$.   
Splitting $u_{k+3}^\geq$ into phase blocks,  we may write 
 \begin{align}\label{eq 2}
\mathrm{I} = \sum_{n\geq 3}^\infty2^{-n}\sum_{k\geq 1}^\infty \int_{\T^d} k^{\alpha}\theta_k^\geq (2^{k+n}u_{k+n}) \cdot \left(\theta*(2^{-k}\nabla\psi_{k-1})\right)\, \dx.
\end{align}
By exploiting the geometric series, it suffices to estimate the sum over $k$ uniformly in $n$. Invoking the Cauchy--Schwarz inequality for the sums and subsequently the H\"older inequality for the integrals, we find that
\begin{align*}
|\mathrm{I}| &\le \int_{\T^d} \Big(\sup_{k\ge 1} k^{\alpha_1 }|\theta_k^\geq |\Big) \Big(\sum_{k\ge 1} |2^{k+n}u_{k+n}|^2 \Big)^{\frac12} \Big(\sum_{k\ge1} k^{2\alpha_2}| \theta*(2^{-k}\nabla\psi_{k-1})|^2\Big)^{\frac12}\, \dx\\
&\le \norm{\sup k^{\alpha_1}|\theta_k^\geq|}_{L^{q}}\norm{\Big(\sum_{k\geq 1} |2^{k+n}u_{k+n}|^2\Big)^{\frac{1}{2}}}_{L^{p}} \norm{\Big(\sum_{k\geq 1}k^{2\alpha_2}|\theta*(2^{-k}\nabla\psi_{k-1})|^2\Big)^\frac{1}{2}}_{L^{s}}.
\end{align*}
We may now use the Littlewood--Paley characterization of Sobolev norms, see, e.g., Theorem 1.3.8 in  \cite{Grafakos_mod}, to estimate the velocity term,
\[
\norm{\Big(\sum_{k\geq 1} |2^{k+n}u_{k+n}|^2\Big)^{\frac{1}{2}}}_{L^{p}} \lesssim \|\grad u\|_{L^p}.
\]
The other two terms on the right-hand side of the above estimate are controlled with the help of the  interpolation inequalities from Lemma \ref{L1}. More precisely, we  apply   \eqref{7}   with $b=\alpha_1=\frac{\alpha}{2}$ and $r=q=\frac{2p}{p-1}$  to the first and \eqref{15} with $r=s=\frac{2p}{p-1}$ and $b=\alpha_2=\frac{\alpha}{2}$ to the third term.  These choices are possible if
\[
\frac{\alpha p}{p-1}<2a,
\]
which is equivalent to $a<\frac{p}{2}$. It follows that the term $\mathrm{I}$ is bounded as in  \eqref{eq 4}.

\medskip

\emph{Estimate of $\mathrm{II}$.} The treatment of $\mathrm{II}$ proceeds very similar to the one of $\mathrm{I}$.  As a first step, we rewrite the term as
\[
\mathrm{II} = \sum_{k\ge 1}k^\alpha \intT  \theta u_{k+3}^{\ge} \cdot \left(\theta_k^\ge * \grad\psi_{k-1}\right) \, \dx,
\]
and we notice by the definition of the multipliers that $\theta_k^{\ge}*\grad\psi_{k-1} = \theta_k\ast \grad\psi_{k-1}$, whose Fourier support concentrates on the annulus $B_{2^{k-1}}(0)\setminus B_{2^{k-2}}(0)$. Moreover, because   the Fourier transform of $u_{k+3}^{\ge}$ vanishes  on $B_{2^{k+1}}(0)$, the small frequency part $\theta_{k}^{\le}$ of $\theta$ does not contribute to $\mathrm{II}$ as can be verified with the help of Parseval's identity.  We may thus write
\[
\mathrm{II}= \sum_{k\ge 1}k^\alpha \intT \theta_{k+1}^{\ge}  u_{k+3}^{\ge} \cdot \theta_k * \grad\psi_{k-1}\, \dx,
\]
which is, in fact,  very similar to $\mathrm{I}$, and, therefore, the proof of the estimate of $\mathrm{II}$ is identical to the one of $\mathrm{I}$.

\medskip

\emph{Estimate of $\mathrm{III}$.} This is a commutator estimate in which the velocity difference is controlled by the velocity gradient.
Considering the inner integral, we have and write
\begin{align*}
\mel \intR\theta_{k+4}^\leq(x-y)(u_{k+2}^\leq(x)-u_{k+2}^\leq(x-y))\nabla\psi_{k-1}(y)\dy\\
& = \int_0^1 \intR\theta_{k+4}^\leq(x-y) \grad u_{k+2}^\leq(x-sy) : \nabla\psi_{k-1}(y)\otimes y\, \dy\ds\\
& =  \int_0^1 \intR\theta_{k+4}^\leq(x-y) \left( \grad u_{k+2}^\leq(x-sy) - \grad u_{k+2}^\leq(x)\right) : \nabla\psi_{k-1}(y)\otimes y\, \dy\ds\\
&\qquad  +  \grad u_{k+2}^\leq(x): \intR\theta_{k+4}^\leq(x-y)    \nabla\psi_{k-1}(y)\otimes y\, \dy  \\
&=:g_k^1(x)+g_k^2(x).
\end{align*}

We first deal with the second contribution $g_k^2$, using some cancellation effects. Indeed, we decompose
\begin{align}
\nabla\psi_{k-1}\otimes y =\Phi_k -I\psi_{k-1},\quad \mbox{where }\Phi_k = \nabla(y\psi_{k-1}),
\end{align}
and notice that $\nabla u_{k+2}^{\leq}:I=\div u_{k+2}^{\leq}=0$ by the incompressibility assumption on $u$, and that $\Phi_k$ is an even function. It follows that we may rewrite 
\[
g_k^2  = \grad u_{k+2}^{\le}: (\theta_{k+4}^{\le }\ast \Phi_k).
\]
Thanks to the fact that the   Fourier transform of $\Phi_k$ vanishes  outside the annulus $B_{2^{k-1}}(0)\backslash B_{2^{k-2}}(0)$, we have $ \theta*\Phi_k=(\theta_k + \theta_{k-1})\ast \Phi_k$. Inserting this information in the full integral, we obtain
 \begin{align}
\sum_{k\geq 1}k^{\alpha}\int_{\T^d} \theta_k^{\ge} g_k^2 \, \dx = \sum_{k\geq 1}k^{\alpha}\int_{\T^d} \theta_{k}^{\geq}\,\nabla u_{k+2}^{\leq}:\left(\theta*\eta_{k-1}^1 *\Phi_k\right)\, \dy,
\end{align}
where we have introduced the mollifiers $\eta_k^j = \phi_k+\dots \phi_{k+j}$. We may now apply Parseval's identity in a similar manner as before and observe that only the frequencies of $\theta_{k}^{\geq}$ that are smaller than $ 2^{k+3}$  contribute to the integral and we may hence replace $\theta_{k}^{\geq}$ by $\theta*\eta_k^4$. We thus arrive at
 \begin{align}
\sum_{k\geq 1}k^{\alpha}\int_{\T^d} \theta_k^{\ge} g_k^2 \, \dx = \sum_{k\geq 1}k^{\alpha}\int_{\T^d} \theta*\eta_{k}^4\,\nabla u_{k+2}^{\leq}:(\theta*\eta_{k-1}^1 *\Phi_k)\, \dy.
\end{align}
Applying Parseval's identity again, 
we see that we may furthermore rewrite this contribution as 
 \begin{align}
\sum_{k\geq 1}k^{\alpha}\int_{\T^d} \theta_k^{\ge} g_k^2 \, \dx & = \sum_{k\geq 1}k^{\alpha}\int_{\T^d} \left(\theta*\eta_{k}^4\right)\,\nabla u:(\theta*\eta_{k-1}^1 *\Phi_k)\, \dy\\
&\qquad  - \sum_{k\geq 1}k^{\alpha}\int_{\T^d} \left(\theta*\eta_{k}^4\right)\,\left(\nabla u\ast \eta_{k+3}^3\right):(\theta*\eta_{k-1}^1 *\Phi_k)\, \dy .
\end{align}
Using the Cauchy--Schwarz inequality in the sum, followed by the H\"older inequality in the integral, we estimate the first term by
\begin{align*}
\mel 
\intT |\grad u| \Big( \sum_{k\ge 1} k^{\alpha} |\theta*\eta_k^4|^2\Big)^{\frac12} \Big(\sum_{k\ge 1} k^\alpha |\theta*\eta_{k-1}^1* \Phi_k|^2\Big)^{\frac12}\, \dx\\ 
& \le \|\grad u\|_{L^p} \|\Big( \sum_{k\ge 1} k^{\alpha} |\theta*\eta_k^4|^2\Big)^{\frac12}\|_{L^q} \|\Big(\sum_{k\ge 1} k^\alpha |\theta*\eta_{k-1}^1* \Phi_k|^2\Big)^{\frac12}\|_{L^s},
\end{align*}
which can be controlled by the right-hand side of \eqref{eq 4} via the interpolation inequality \eqref{15} of Lemma \ref{L1} with $r=q=s$ and $b=\alpha$. 
Applying the same tools to the second term, we find the bound
\begin{align}
\mel \int_{\T^d} \Big(\sup_{k\ge 1} k^{\alpha_1} |\theta*\eta_{k}^4|\Big) \Big( \sum_{k\geq 1} |\nabla u\ast \eta_{k+3}^3|^2\Big)^{\frac12}\Big(\sum_{k\geq 1} k^{2 \alpha_2}|\theta*\eta_{k-1}^1 *\Phi_k|^2\Big)^{\frac12}\, \dy\\
&\le  \| \Big(\sum_{k\ge 1}|\nabla u*\eta^3_k|^2\Big)^{\frac12}\|_{L^p} \|\sup_{k\ge 1} k^{\alpha_1} |\theta*\eta_{k}^4|\|_{L^q} \|\Big(\sum_{k\geq 1} k^{2 \alpha_2}|\theta*\eta_{k-1}^1 *\Phi_k|^2\Big)^{\frac12}\|_{L^s},
\end{align}
which is, analogously to the estimates for $\mathrm{I}$ and $\mathrm{II}$, controlled by using the interpolations \eqref{7a} and \eqref{15} from Lemma \ref{L1} and the Littlewood--Paley theorem \eqref{28}.

We finally turn to the $g^1_k$ term. Noticing that \begin{align}
&\intR \left( \grad u_{k+2}^\leq(x-sy) - \grad u_{k+2}^\leq(x)\right) : \nabla\psi_{k-1}(y)\otimes y\, \dy\ds\\
=&\intR \psi_{k-1}\div_y(y\cdot\nabla u_{k+2}^\leq(x-sy))-\scalar{\div_y u_{k+2}^\leq(x-sy)}{\nabla\psi_{k-1}(y)}\, \dy\ds=0
\end{align} we may add an extra $\theta_{k+4}^\leq(x)$ in $g_k^1$ and obtain 
\begin{align*}
  g_k^1(x) 
& =\int_0^1 \intR\left(\theta_{k+4}^\leq(x-y) -\theta_{k+4}^\le(x)\right)\\
 & \qquad \qquad \times \left( \grad u_{k+2}^\leq(x-sy) - \grad u_{k+2}^\leq(x)\right) : \nabla\psi_{k-1}(y)\otimes y\, \dy\ds.
\end{align*}
We can further reformulate this	 by expressing the differences in term of  the mean values,
\begin{align*}
 g_k^1(x) 
& =  \int_0^1\int_0^1\int_0^1 \intR y\cdot \grad \theta_{k+4}^\leq(x-ry)  
 \grad^2 u_{k+2}^\leq(x-st  y)   : \nabla\psi_{k-1}(y)\otimes y\otimes y\, \dy\dr\ds\dt.
\end{align*}
Splitting $\theta$ and $u$ into phase blocks, and using the convention that $\theta_{\ell} = 0$ and $u_m=0$ for $\ell,m\le 0$, the latter is furthermore controlled by
\begin{align*}
 |g_k^1(x) |
&\le \sum_{j\ge -4} \sum_{n\ge -2} \int_0^1\int_0^1\int_0^1 \intR |\grad\theta_{k-j}(x-ry)|\\
&\qquad \qquad \qquad \times  |\grad^2 u_{k-n}(x-sty)|   |\grad \psi_{k-1}(y)||y|^3\, \dy\dr\ds\dt\\
& = \sum_{j\ge -4} 2^{-j}\sum_{n\ge -2}2^{-n} \int_0^1\int_0^1\int_0^1\intR 2^{j-k}|\grad\theta_{k-j}(x-ry)|\\
&\qquad \qquad \qquad \times 2^{n-k} |\grad^2 u_{k-n}(x-sty)|   \rho_k(y)\, \dy\dr\ds\dt,
\end{align*}
where we have introduced $\rho_k(y) = 4^k|y|^3 |\grad\psi_{k-1}(y)|$ for notational convenience. Multiplying by $k^{\alpha_2}$, summing over $k$ and using the H\"older inequality then gives
\begin{align*}
\sum_{k\ge 1}  k^{\alpha_2} |g_k^1(x)| &\le  \sum_{j\ge -4} 2^{-j}\sum_{n\ge -2}2^{-n} \int_0^1\int_0^1\int_0^1\intR \Big( \sum_{k\ge 1} k^{2\alpha_2} 2^{2(j-k)}|\grad\theta_{k-j}(x-ry)|^2\Big)^{\frac12}\\
&\qquad \qquad \qquad \times\Big(\sum_{k\ge 1}  2^{2(n-k)} |\grad^2 u_{k-n}(x-sty)|^2\Big)^{\frac12}   \rho_k(y)\, \dy\dr\ds\dt,
\end{align*}
and therefore, integrating against $k^{\alpha_1} \theta_k^{\ge}$ and using the H\"older inequality in the integrals, we deduce that
\begin{align*}
\mel \sum_{k\ge 1}  k^{\alpha} \intT |\theta_k^{\ge}||g_k^1| \, \dx\\
&\le  \sum_{j\ge -4} 2^{-j}\sum_{n\ge -2}2^{-n} \int_0^1\int_0^1\int_0^1\intR  \|\sup_{k\ge 1} k^{\alpha_1}|\theta_k^{\ge}|\|_{L^q} \\
 &\qquad \qquad \qquad\times\| \Big( \sum_{k\ge 1} k^{2\alpha_2} 2^{2(j-k)}|\grad\theta_{k-j}(\cdot-ry)|^2\Big)^{\frac12}\|_{L^s}\\
&\qquad \qquad \qquad \times \|\Big(\sum_{k\ge 1}  2^{2(n-k)} |\grad^2 u_{k-n}(\cdot-sty)|^2\Big)^{\frac12} \|_{L^p}  \rho_k(y)\, \dy\dr\ds\dt.
\end{align*}
We make now make use of the periodicity of the problems and the convergence of the geometric series to deduce 
\begin{align*}
\sum_{k\ge 1}  k^{\alpha} \intT |\theta_k^{\ge}||g_k^1| \, \dx
& \le \|\sup_{k\ge 1} k^{\alpha_1}|\theta_k^{\ge}|\|_{L^q} \| \Big( \sum_{\ell\ge 1} \ell^{2\alpha_2} 2^{-2\ell}|\grad\theta_{\ell} |^2\Big)^{\frac12}\|_{L^s}\\
&\qquad \qquad\times  \|\Big(\sum_{m\ge 1}  2^{-2m} |\grad^2 u_{m} |^2\Big)^{\frac12} \|_{L^p} \| \rho_k\|_{L^1}.
\end{align*}
The estimate of the right-hand side proceeds as before via Lemma \ref{L1}. Notice only that we can write $2^{-\ell}\grad\theta_{\ell} = \theta\ast\eta_{\ell}$ with $\eta_{\ell} = 2^{-\ell}\grad\phi_{\ell}$ satisfying the hypothesis of the lemma, and that $\|\rho_k\|_{L^1} = \|\rho_1\|_{L^1}\sim 1$ due to scaling.

This completes the estimate of the term $\mathrm{III}$ and finishes the proof.
\end{proof}

\section{Estimates for the zero-diffusivity limit. Proof of Theorem \ref{T4}}\label{S4}

We finally turn to the proof of Theorem \ref{T4}. It makes use of the following interpolation result.
\begin{lemma}\label{L3}
For any $\ell \ge 2$, it holds that
\[
\|\grad \theta\|_{L^2} \lesssim \ell\left(\frac1{\log\ell}\right)^{a}  \| \theta\|_{B^{\log,a}}+ \left(\frac1{\log\ell}\right)^{a} \|\grad \theta\|_{B^{\log,a}}.
\]
\end{lemma}

\begin{proof}
By interpolation, it is enough to consider $\ell=2^K$ for some arbitrary $K\in\N$. Using the Littlewood--Paley characterization of Sobolev norms \eqref{24}, we may decompose the gradient norm as
\[
\|\grad \theta\|_{L^2}^2  \sim \sum_{k\le K_0} \|\grad \theta\ast\phi_k\|_{L^2}^2 +  \sum_{K_0+1\le k\le K} \|\grad \theta\ast\phi_k\|_{L^2}^2+\sum_{k\ge K+1} \|\grad \theta\ast\phi_k\|_{L^2}^2,
\]
where  $K_0 $ is such that the mapping $k\mapsto \frac{4^k}{k^{2a}}$ is increasing for $k\ge K_0$, therefore, $K_0 = \frac{2a}{\log 4}$, and understanding that the middle term vanishes if $K\le K_0$.  For the first term, we notice that, by the boundedness of the considered frequencies, it holds
\begin{align*}
\sum_{k\le K_0} \|\grad \theta\ast\phi_k\|_{L^2}^2&  \lesssim 4^{K_0} \sum_{k\ge 1} \| \theta\ast\phi_k\|_{L^2}^2 \lesssim \sum_{k\ge 1} k^{2a}\| \theta\ast\phi_k\|_{L^2}^2\lesssim  \| \theta\|_{B^{\log,a}}^2.
\end{align*}
Moreover, by the choice of $K_0$, the second term can be bounded using monotonicity,
\[
 \sum_{K_0+1\le k\le K} \|\grad \theta\ast\phi_k\|_{L^2}^2 \lesssim \frac{4^K}{K^{2a}} \sum_{K_0+1\le k\le K} k^{2a}\| \theta\ast\phi_k\|_{L^2}^2 \le \frac{4^K}{K^{2a}}\| \theta\|_{B^{\log,a}}^2,
\]
while for the third term, we use the lower bound on the considered frequencies,
\[
\sum_{k\ge K+1} \|\grad \theta\ast\phi_k\|_{L^2}^2 \lesssim \frac1{K^{2a}}\sum_{k\ge K+1} k^{2a}\|\grad \theta\ast\phi_k\|_{L^2}^2  \le  \frac1{K^{2a}} \| \grad \theta\|_{B^{\log,a}}^2.
\]

Combining these bounds yields the statement.
\end{proof}

With these preparations, we are in the position to establish the bound on the dissipation term in Theorem \ref{T4},
 \begin{equation}\label{26a}
\begin{aligned}
 \left(\kappa \int_0^t \|\grad\thetak\|_{L^2}^2\, \dt\right)^{1/2} \le \frac{C}{\log^a\left(2+\frac1{\kappa t}\right)}\left(  \left(\int_0^t \|\grad u\|_{L^p}\, \dt\right)^a \|\theta_0\|_{L^{\infty}} + \|\theta_0\|_{B^{\log, a}}\right).
 \end{aligned}
\end{equation}

\begin{proof}[Proof of Theorem \ref{T4}. Part 1: Gradient estimate \eqref{26a}]
The statement holds trivially for all sufficiently big $\kappa t$ by making the implicit constant in the inequality big enough, hence we may assume that $\kappa t$ is small, say $\leq \frac{1}{100}$.

From the interpolation estimate of Lemma \ref{L3} and the regularity estimates of Theorem \ref{T2}, we obtain that
\begin{align*}
\kappa \int_0^t \|\grad \theta\|_{L^2}^2\, \ds &\lesssim \frac{\kappa \ell^2}{(\log\ell)^{2a}}\int_0^t \|\theta\|_{B^{\log,a}}^2\, \ds + \frac{\kappa}{(\log\ell)^{2a}} \int_0^t \|\grad\theta\|_{B^{\log,a}}^2\,\ds\\
&\lesssim \frac{\kappa t\ell^2 +1}{(\log \ell)^{2a}} \left(\|\theta_0\|_{B^{\log,a}}^2 + \left(\int_0^t \|\grad u\|_{L^p}\, \ds\right)^{2a}\|\theta_0\|_{L^{\infty}}^2\right),
\end{align*}
which holds true for any $\ell\ge2$.  Picking $\ell\sim\frac1{\sqrt{t\kappa}}$ yields the desired result for $\kappa t$ small enough.
\end{proof}

It now remains to establish the estimate on the convergence rate stated in Theorem~\ref{T4},
\begin{equation}\label{26b}
\begin{aligned}
\mel \|\theta(t) - \thetak(t)\|_{L^2}  \\
& \lesssim  \log^{-a}\left(2+\frac1{\kappa t}\right) \left(\left(1+ \left(\int_0^t \|\grad u\|_{L^p}\, \dt\right)^p\right) \|\theta_0\|_{L^{\infty}} + \|\theta_0\|_{B^{\log, a}}\right).
 \end{aligned}
\end{equation}
Its derivation requires some further preparations. In a first step, we derive a sharp estimate on the rate of \emph{weak convergence} for the zero-diffusivity limit. Studying weak convergence is convenient, as a direct estimate on the $L^2$-distance would require sharp smoothing estimates to control the norm of the Laplacian in terms of the Besov norm of the initial data --- which are not even known to exists. Here, weak convergence is measured in terms of Kantorovich--Rubinstein distances with logarithmic cost,  which were introduced earlier in \cite{Seis18,Seis17}: For two nonnegative measures $\mu,\nu\in \mathcal{M}(\T^d)$ with $\mu[\T^d]=\nu[\T^d]$ and a constant $\delta>0$,  we define 
\begin{align}
\mathcal{D}_\delta(\mu,\nu)=\inf_{\pi\in \Pi(\mu,\nu)}\int_{\T^{2d}}\log\left(\frac{|x-y|}{\delta}+1\right)  \,  \text{d}\pi(x,y),
\end{align}
where $\Pi(\mu,\nu)$ is the set of all couplings of the measures $\mu$ and $\nu$. We refer to Villani's monograph \cite{Villani03} for a comprehensive introduction into the theory of optimal transportation. The cost function
\[
c_{\delta}(z) = \log\left(\frac{z}{\delta}+1\right)  
\]
that we consider here is concave and increasing and gives thus rise to a metric on $\T^d$, $d_{\delta}(x,y) = c_{\delta}(|x-y|)$.  Thanks to the Kantorovich--Rubinstein duality, see, e.g., Theorem 1.14 in \cite{Villani03}, 
\[
\mathcal{D}_{\delta}(\mu,\nu)  = \sup\left\{\intT \phi\, \dd(\mu-\nu):\: |\phi(x)-\phi(y)|\le d_{\delta}(x,y)\right\}, 
\]
the optimal transportation problem  is in fact a transshipment problem and the Kantorovich--Rubinstein distance extends to a norm on mean-zero measures $\sigma \in \mathcal{M}(\T^d)$,
\[
\D_{\delta}(\sigma) = \D_{\delta}(\sigma^+,\sigma^-),
 \]
 where the superscripted plus and minus signs indicate the positive and negative parts, respectively. It is thus suitable to study the convergence of $\thetak$ towards $\theta$ with respect to this distance. This was done previously in \cite{Seis18} in a setting in which $\theta$ has no regularity properties. Here, we  obtain a small improvement thanks the regularity established in Theorems \ref{T1} and~\ref{T2}.

\begin{proposition}\label{L6}
Under the hypotheses of Theorems \ref{T1} and \ref{T2}, there exists a constant $C$ dependent on $d$, $p$ and $a$   such that
\begin{align}
\mathcal{D}_{\delta}(\theta(t),\theta^{\kappa}(t))&\lesssim\sup_{s\le t}\norm{\theta(s)-\thetak(s)}_{L^q}\int_0^t\norm{\nabla u}_{L^p}\, \dt\\
&\quad + \frac1{\delta} \frac{\sqrt{\kappa t}}{\log^a\left( 2+\frac{1}{\kappa t}\right)}\left(\left(\int_0^t\norm{\nabla u}_{L^p}\, \dd s\right)^a\norm{\theta_0}_{L^\infty}+\norm{\theta_0}_{B^{\log,a}}\right),
\end{align}
for any $t>0$,   where $\frac{1}{p}+\frac{1}{q}=1$.
\end{proposition}

As the first term on the right-hand side is bounded by the initial datum via \eqref{23} and \eqref{30}, the logarithmic distance is bounded uniformly in $\kappa t$ provided that 
\begin{equation}
\label{26c}
\delta \ge \delta(t):=\frac{\sqrt{\kappa t}}{\log^{a}\left(2+\frac{1}{\kappa t}\right)}.
\end{equation}
Proposition \ref{L6} thus proves  that $\thetak(t)$ converges to $\theta(t)$ weakly with rate $\delta(t)$. The case $a=0$ describes the standard rate of convergence observed in the zero-diffusivity limit without regularity assumptions, cf.~\cite{Seis18}.

\begin{proof}
Our starting point is the following estimate on the rate of change of the logarithmic distance,
\begin{align}
\frac{\text{d}}{\text{d}t}\mathcal{D}_\delta(\theta(t),\theta^\kappa(t))\lesssim\norm{\nabla u(t)}_{L^p}\norm{\theta(t)-\thetak(t)}_{L^q}+ \frac{\kappa}{\delta}\norm{\nabla \thetak(t)}_{L^1},
\end{align}
which was derived earlier in \cite{Seis18} and which holds true for any $\delta>0$. Integrating in time,  and using   the assumption that $\theta$ and $\thetak$ have both the same intial datum so that their logarithmic distance vanishes initially, we find that
 \begin{align}
\mathcal{D}_\delta(\theta(t),\theta^\kappa(t))\lesssim\sup_{s\le t}\norm{\theta(s)-\thetak(s)}_{L^q}\int_0^t\norm{\nabla u}_{L^p}\, \dt+ \frac{\kappa}{\delta}\int_0^t\norm{\nabla \thetak}_{L^1}\ds.
\end{align}
We now invoke Jensen's inequality and our gradient bound \eqref{26a} to estimate the dissipation term:
\begin{align}
\frac{\kappa}{\delta}\int_0^t\norm{\nabla \thetak}_{L^1}\ds & \lesssim \frac{\sqrt{\kappa t}}{\delta}\left(\kappa \int_0^t\norm{\nabla \thetak}_{L^2}^2\ds\right)^{\frac{1}{2}}\\
& \lesssim \frac1{\delta} \frac{\sqrt{\kappa t}}{\log^a\left( 2+\frac{1}{\kappa t}\right)}\left(\left(\int_0^t\norm{\nabla u}_{L^p}\, \dd s\right)^a\norm{\theta_0}_{L^\infty}+\norm{\theta_0}_{B^{\log,a}}\right).
\end{align}
Combining the previous two bounds gives the statement of the proposition.
\end{proof}

In order to translate the estimate on weak convergence into an estimate on strong convergence, we have to make use of an interpolation inequality.

\begin{lemma}\label{L5}
Let $\sigma$ be a mean-zero function in $ L^1(\T^d)\cap B^{\log,a}(\T^d)$ for some 
$a>0$. Then it holds for  any $\ell\ge 2$ that
 \begin{align}
\norm{\sigma}_{L^1}\lesssim   \frac{\mathcal{D}_\delta(\sigma )}{c_{\delta}(1/\ell)}  +\left(\frac{1}{\log \ell}\right)^{ a}\norm{\sigma}_{B^{\log,a}}.
\end{align}

\end{lemma}
The estimate extends previous interpolations between logarithmic Kantorovich--Ru\-bin\-stein distances and Sobolev norms \cite{BrenierOttoSeis11,OttoSeisSlepcev13} to our Besov space setting, and belongs to a family of Kantorovich--Sobolev inequalities \cite{Ledoux15}.
\begin{proof}
Setting $k:=\lfloor \log_2 \ell\rfloor+1$, we can rewrite the $L^1$-norm as \begin{align}
\norm{\sigma}_{L^1}=\sup_{\|\rho\|_{L^{\infty}}\le 1} \intT\sigma\rho\, \dx =\sup_{\|\rho\|_{L^{\infty}}\le 1} \left( \int_{\T^d}\sigma_{k+1}^\geq\rho\, \dx + \int_{\T^d} \sigma_k^\leq\rho\, \dx\right).
\end{align}
We will estimate both summands separately and keep $\rho$ with $\|\rho\|_{L^{\infty}}\le 1$ arbitrarily fixed for a moment.

We first estimate the first summand. Using Jensen's inequality, the Littlewood--Paley decomposition for $\sigma$ and  the fact that the support of any $\widehat{\sigma_k}$  intersects precisely with three frequency blocks \eqref{27}, we have that
\begin{align}
\int_{\T^d} \sigma_{k+1}^\geq \rho\, \dx\le \|\sigma_{k+1}^{\ge}\|_{L^1} \lesssim \|\sigma_{k+1}^{\ge}\|_{L^2}\le  \Big(\sum_{j\geq k}\norm{\sigma_j}_{L^2}^2\Big)^{1/2}.
\end{align}
In view of our choice of $k$, the latter is contolled as follows: 
\begin{align}
\int_{\T^d} \sigma_{k+1}^\geq\rho \, \dx\lesssim \frac{1}{k^{a}}\Big(\sum_{j\geq k} j^{2a}\norm{\sigma_j}_{L^2}^2\Big)^{1/2} 
\lesssim \left(\frac{1}{\log \ell}\right)^{a}\norm{\sigma}_{B^{\log,a}}.\label{eq27}
\end{align}

For the second summand, we choose an arbitrary coupling $\pi$ between $\sigma^+$ and $\sigma^-$ and write
\begin{align}
&\int_{\T^d} \sigma_k^\leq\rho \, \dx=\int_{\T^d} \sigma \rho_k^\leq \, \dx= \int_{\T^d}(\sigma^+-\sigma^-)\rho_k^\leq\, \dx=\int_{\T^{2d}} \rho_k^\leq(x)-\rho_k^\leq(y)\, \text{d}\pi(x,y).
\end{align}
Given some $r>0$ that will be fixed later, we consider separately diagonal contributions $|x-y|\le r$ and off-diagonal contributions $|x-y|>r$. 
For the diagonal contributions we use a Lipschitz estimate,
\begin{align}
&\int_{|x-y|\leq r}\rho_k^\leq(x)-\rho_k^\leq(y)\, \text{d}\pi(x,y)\leq \norm{\nabla \rho_k^\leq}_{L^\infty}r \int_{|x-y|\leq r}\text{d}\pi(x,y)\leq r\norm{\nabla \rho_k^\leq}_{L^\infty}\norm{\sigma}_{L^1},
\end{align}
and bound the gradient with the help of the convolution estimate and by scaling,
\[
\norm{\nabla \rho_k^\leq}_{L^\infty} = \norm{ \rho \ast\grad\psi_k}_{L^\infty} \le \norm{\grad\psi_k}_{L^1} \norm{ \rho  }_{L^\infty}  \lesssim 2^k 
\]
to the effect that
\begin{align}
&\int_{|x-y|\leq r}\rho_k^\leq(x)-\rho_k^\leq(y)\, \text{d}\pi(x,y)\lesssim  \ell r\norm{\sigma}_{L^1},\label{eq26}
\end{align}
by our choice of $k$.
For the off-diagonal term we use the monotonicity of the cost function to bound
\begin{align}
\int_{|x-y|>r}\rho_k^\leq(x)-\rho_k^\leq(y)\, \text{d}\pi(x,y)
&\lesssim  \frac{\norm{\rho}_{L^\infty}}{c_{\delta}(r)}   \int_{|x-y|>r}c_{\delta}(|x-y|)\, \text{d}\pi(x,y) \le  \frac{ \D_{\delta}(\sigma)}{c_{\delta}(r)}.\label{eq28}
\end{align}

Adding \eqref{eq28}, \eqref{eq26} and \eqref{eq27} and choosing $r=\frac1{C\ell}$ for some $C>0$, we obtain the statement.\end{proof}

It remains to combine the Proposition \ref{L6} and Lemma \ref{L5} to deduce the desired estimate on the rate of strong convergence.

\begin{proof}[Proof of Therem \ref{T4}. Part 2: Convergence rate \eqref{26b}.]
Again, we may assume without loss of generality that $\kappa t$ is small, say $\leq \frac{1}{100}$, since otherwise the statement follows from the fact that $\norm{\theta-\theta^\kappa}_{L^2}\lesssim \norm{\theta_0}_{L^\infty}$, cf.~\eqref{23} and \eqref{30}, by making the constant in the inequality big enough.

We start with two observations. On the one hand, by choosing
$\delta =\sqrt{\kappa t}$, and taking the supremum in $s\leq t$ the estimate in Proposition \ref{L6} yields
\begin{align}
\sup_{s\le t}\D_{\sqrt{\kappa s}}(\theta(s),\thetak(s)) & \lesssim  \sup_{s\le t} \|\theta(s)-\thetak(s)\|_{L^q} \int_0^t \|\grad u\|_{L^p}\, \dt + \left(\log^{-a}\frac1{\kappa t}\right) \Lambda_{\theta_0,u,a}(t),
\end{align}
where
\[
\Lambda_{\theta_0,u, a}(t) =   \left(\int_0^t \|\grad u\|_{L^p}\, \dt\right)^a \|\theta_0\|_{L^{\infty}} + \|\theta_0\|_{B^{\log,a}}.
\]
By interpolation in Lebesgue spaces and thanks to the a priori estimates in \eqref{23} and \eqref{30}, the latter implies
\begin{align}\label{28a}
\mel \sup_{s\le t}\D_{\sqrt{\kappa s}}(\theta(s),\thetak(s))\\
& \lesssim   \sup_{s\le t} \|\theta(s)-\thetak(s)\|_{L^1}^{1/q}\|\theta_0\|_{L^{\infty}}^{1-1/q} \int_0^t \|\grad u\|_{L^p}\, \dt+ \left(\log^{-a}\frac1{\kappa t}  \right)\Lambda_{\theta_0,u,a}(t),
\end{align} On the other hand, the interpolation estimate in  Lemma \ref{L5}, with $\delta=\sqrt{\kappa t}$ and $\ell = (\kappa t)^{-1/4}$ gives
\begin{align}\label{28b}
\norm{\theta(t)-\theta^\kappa(t)}_{L^1}
&\lesssim 
\left(  \log^{-1}\frac1{\kappa t}\right)\D_{\sqrt{\kappa t}}(\theta(t),\thetak(t)) + \left(\log^{-a}\frac1{\kappa t}\right) \Lambda_{\theta_0,u,a}(t).
\end{align}
We plug this bound into \eqref{28a} and use the elementary estimate $xy  = \eps^{1/q}x\eps^{-1/q}y\le \frac{\eps}{q} x^q + \frac1{\eps^{p/q}p}y^p$,
\begin{align*}
\mel\sup_{s\le t}\D_{\sqrt{\kappa s}}(\theta(s),\thetak(s))\\
& \lesssim \left(\log^{-1/q}\frac1{\kappa t}\right)\|\theta_0\|_{L^{\infty}}^{1-1/q}\int_0^t \|\grad u\|_{L^p}\, \dt\left(\sup_{s\le t}\D_{\sqrt{\kappa s}}(\theta(s),\thetak(s))\right)^{1/q}\\
&\quad + \left(\log^{-a/q}\frac1{\kappa t}\right) \|\theta_0\|_{L^{\infty}}^{1-1/q}\int_0^t \|\grad u\|_{L^p}\, \dt\left(\Lambda_{\theta_0,u,a}(t)\right)^{1/q}+ \left(\log^{-a}\frac1{\kappa t}  \right)\Lambda_{\theta_0,u,a}(t) \\
&\lesssim \eps \left(\sup_{s\le t}\D_{\sqrt{\kappa s}}(\theta(s),\thetak(s))\right) 
+ \eps^{-p/q}\left(\log^{-p/q}\frac1{\kappa t}\right)\|\theta_0\|_{L^{\infty}}\left(\int_0^t \|\grad u\|_{L^p}\, \dt\right)^{p}\\
&\quad + \left(\log^{-a/q}\frac1{\kappa t}\right) \|\theta_0\|_{L^{\infty}}^{1-1/q}\int_0^t \|\grad u\|_{L^p}\, \dt\left(\Lambda_{\theta_0,u,a}(t)\right)^{1/q}+ \left(\log^{-a}\frac1{\kappa t}  \right)\Lambda_{\theta_0,u,a}(t).
\end{align*}
The first term on the right-hand side can be absorbed on the left-hand side if $\eps$ is chosen sufficiently small. Since $a/q<\min\{a,p/q\}$ by our choices of $a$, $p$, and $q$, obtain after using Young's inequality on the norms and integrals that
\[
\sup_{s\le t}\D_{\sqrt{\kappa s}}(\theta(s),\thetak(s)) \lesssim \left(\log^{-a/q}\frac1{\kappa t}\right) \left(\|\theta_0\|_{L^{\infty}}\left(\int_0^t\|\grad u\|_{L^p}\, \dt\right)^p+\Lambda_{\theta_0,u,a}(t)\right).
\]
Inserting  this estimate for our logarithmic distance into the estimate on the convergence rate \eqref{28b}, we obtain
\begin{align}
\mel \sup_{s\le t}\norm{\theta(s)-\theta^\kappa(s)}_{L^1}
\\
&\lesssim 
\left(\log^{-1-a/q}\frac1{\kappa t}\right) \left(\|\theta_0\|_{L^{\infty}}\left(\int_0^t\|\grad u\|_{L^p}\, \dt\right)^p+\Lambda_{\theta_0,u,a}(t)\right)\\
&\quad+ \left(\log^{-a}\frac1{\kappa t}\right) \Lambda_{\theta_0,u,a}(t).
\end{align}
We note that $-a\geq -1-\frac{a}{q}$ precisely if $a\leq p$, which is always true under our assumption. Hence we obtain 
\begin{align}
\mel \sup_{s\le t}\norm{\theta(s)-\theta^\kappa(s)}_{L^1}
\\
&\lesssim 
\left(\log^{-a}\frac1{\kappa t}\right) \left(\|\theta_0\|_{L^{\infty}}\left(\int_0^t\|\grad u\|_{L^p}\, \dt\right)^p+\Lambda_{\theta_0,u,a}(t)\right).
\end{align}
This easily implies \eqref{26b} after inserting the defintion of $\Lambda_{\theta_0,u,a}$.
\end{proof}

\bibliography{euler}
\bibliographystyle{abbrv}
\end{document}